\documentclass[12pt]{amsart}

\usepackage{latexsym}
\usepackage{amssymb}
\usepackage[cp850]{inputenc}
\usepackage{amsthm}
\usepackage{amscd}
\usepackage{amsmath}
\usepackage{amsfonts}
\usepackage[all]{xy}
\usepackage[pdftex]{color}
\usepackage[pdftex]{graphicx}

\newtheorem{sat}{Theorem}				\newtheorem{lem}{Lemma}
\newtheorem{kor}{Corollary}			\newtheorem{prop}{Proposition}
				
\newtheorem*{defi*}{Definition}	\newtheorem*{bei*}{Example}
\newtheorem*{sat*}{Theorem}			\newtheorem*{kor*}{Corollary}
\newtheorem*{rmk*}{Remark}
\newtheorem*{conj*}{Conjecture}

\let\ssection=\section
\renewcommand{\section}{\setcounter{equation}{0}\ssection}

\newtheorem*{namedtheorem}{\theoremname}
\newcommand{\theoremname}{testing}
\newenvironment{named}[1]{\renewcommand{\theoremname}{#1}\begin{namedtheorem}}{\end{namedtheorem}}

\theoremstyle{remark}

\newcommand{\BR}{\mathbb R}			
\newcommand{\BN}{\mathbb N}			
\newcommand{\BS}{\mathbb S}			
				
\newcommand{\BP}{\mathbb P}

		\newcommand{\CL}{\mathcal L}

		\newcommand{\CX}{\mathcal X}

\newcommand{\D}{\partial}

\DeclareMathOperator{\inj}{inj}
\DeclareMathOperator{\diam}{diam}

\DeclareMathOperator{\interior}{interior}
\DeclareMathOperator{\cut}{cut}

\DeclareMathOperator{\Conj}{Conj}
\DeclareMathOperator{\length}{length}
\begin{document}

\title[]{A characterization of round spheres in terms of blocking light}
\author{Benjamin Schmidt \& Juan Souto}
\begin{abstract}
A closed Riemannian manifold $M$ is said to have cross (compact rank one symmetric space) blocking if whenever $p\neq q$ are less than the diameter apart, all light rays from $p$ can be shaded away from $q$ with at most two point shades.  Similarly, a closed Riemannian manifold is said to have sphere blocking if for each $p \in M$ all the light rays from $p$ are shaded away from $p$ by a single point shade.  We prove that Riemannian manifolds with cross and sphere blocking are isometric to round spheres.  
\end{abstract}
\maketitle

In this note we characterize constant curvature spheres in terms of light blocking properties. 

 \begin{defi*}[Light]
Let $X,Y$ be two nonempty subsets of a Riemannian manifold $M$, and let $G_M(X,Y)$ denote the set of non-constant unit speed parametrized geodesics  $\gamma:[0,L_{\gamma}] \to M$ with initial point $\gamma(0) \in X$ and terminal point $\gamma(L_{\gamma})\in Y$.  The {\em light from X to Y} is the set 
$$L_{M}(X,Y)=\{\gamma \in G_{M}(X,Y) | \interior(\gamma) \cap (X \cup Y)= \emptyset \}.$$  
A subset $Z\subset M$ {\em blocks} the light from $X$ to $Y$ if  the interior of every $\gamma\in L_M(X,Y)$ meets $Z$.
\end{defi*}

Intuitively, we are postulating that $X$ emits light traveling along geodesics, that $Y$ consists of receptors, and that $X$ and $Y$ are opaque while the remaining medium $M\setminus\{X\cup Y\}$ is transparent.  From this point of view, $L_M(X,Y)$ is the set of light rays from $X$ to $Y$ and a set $Z$ blocks the light from $X$ to $Y$ if it completely shades $X$ away from $Y$. This simple model ignores diffraction, the dual nature of light, and all aspects of quantum mechanics.

A well known result of Serre \cite{Se} asserts that for compact $M$ and points $x,y \in M$, the set $G_M(x,y)$ of geodesic segments joining $x$ and $y$ is always infinite. In contrast, $L_M(x,y)$ is sometimes infinite and sometimes not. For instance, if $x$ and $y$ are different points on the standard round sphere $\BS^n$ with distance less than $\pi$, then $L_{\BS^n}(x,y)$ consists of exactly two elements. In particular, we see that, under the same assumptions, it suffices to declare two additional points in $\BS^n$  to be opaque in order to block all the light rays from $x$ to $y$.

\begin{defi*}[Blocking Number]
Let $x,y\in M$ be two (not necessarily distinct) points in $M$.  The \textit{blocking number} $b_{M}(x,y)$ for $L_{M}(x,y)$ is defined by 
$$ b_{M}(x,y)=\inf \{n \in \mathbb{N} \cup \{\infty\} \vert L_{M}(x,y) \textup{ is blocked by $n$ points} \}. $$
\end{defi*}

The study of blocking light (also known as security) seems to have originated in the study of polygonal billiard systems and translational surfaces (see e.g. \cite{Fo}, \cite{Gu1}, \cite{Gu2}, \cite{Gu3}, \cite{HiSn}, \cite{Mo1}, \cite{Mo2}, \cite{Mo3}, \cite{Mo4}, and \cite{Ta}). More recently, blocking light has been studied in Riemannian spaces (see e.g. \cite{BuGu}, \cite{Gu4}, \cite{GuSc}, \cite{He}, and \cite{LaSc}). Here we give a characterization of the round sphere in terms of its blocking properties.

If $x,y$ are two distinct points in the standard round sphere $\BS^n$ closer than $\pi$ then, as remarked above, $b_{\BS^n}(x,y)\le 2$. This property does not characterize the round sphere amongst all closed Riemannian manifolds. In fact, every compact rank one symmetric space, or CROSS for short, has the following property:
\begin{itemize}
\item[${ }$] {\em Cross blocking}: For every distinct pair of points $x,y\in M$ with $d_M(x,y)<\diam(M)$, we have $b_M(x,y)\le 2$.
\end{itemize}
Apart from cross blocking, the round sphere also has the following property:
\begin{itemize}
\item[${ }$]{\em Sphere blocking}: For every point $x\in M$, we have $b_M(x,x)=1$.
\end{itemize}

\vspace{0.2cm}
The CROSSes are classified and consist of the round spheres $\BS^n$, the projective spaces $K\BP^n$ where $K$ denotes one of $\mathbb{R}$, $\mathbb{C}$, or $\mathbb{H}$, and the Cayley projective plane, each one endowed with its symmetric metric. It is not difficult to check that the round sphere is the only CROSS with sphere blocking.

In \cite{LaSc} it was conjectured that a closed Riemannian manifold with cross and sphere blocking is isometric to a round sphere. We prove that this is the case:

\begin{sat}\label{thm:sphere}
A closed Riemannian manifold $M$ has cross and sphere blocking if and only if $M$ is isometric to a round sphere.
\end{sat}

In order to prove Theorem \ref{thm:sphere} we show that manifolds as in the statement are Blaschke manifolds. Recall that a compact Riemannian manifold $M$ is said to be {\em Blaschke} if its injectivity radius and diameter coincide. Berger \cite{Be} proved that a Blaschke manifold diffeomorphic to the sphere is in fact isometric to a round sphere.  This was used in \cite{LaSc} to prove Theorem \ref{thm:sphere} for Blaschke manifolds.  
 \vspace{0.2cm}
 
 In \cite{LaSc} it was also conjectured that a closed Riemannian manifold with cross blocking is isometric to a compact rank one symmetric space.  We prove that this is the case in dimension two:
 
 \begin{sat}\label{thm:surfaces}
 A closed Riemannian surface $M$ has cross blocking if and only if $M$ is isometric to a constant curvature sphere or projective plane. 
 \end{sat}

Section \ref{sec:preli} contains some preliminary material concerning Morse theory for path spaces and properties of totally convex subsets in Riemannian manifolds.  In section \ref{sec:light} we prove Theorems \ref{thm:sphere} and \ref{thm:surfaces}.
\vspace{0.2cm}

\textbf{Acknowledgements}  The first author was partially funded by an NSF Postdoctoral Fellowship during the period this work was completed.  He thanks the FIM Institute for Mathematical Research for its hospitality during the earlier stages of writing.  The second author would like to thank the Department of Mathematics of Stanford University for its hospitality while most of this paper was being written.

\section{Preliminaries}\label{sec:preli}
In order to fix notation we start reviewing some well known definitions and results in differential geometry.  We then review the basic aspects about Morse theory on path spaces and about totally convex subsets in Riemannian manifolds needed in Section 2.  Good references for this material include Milnor's \textit{Morse Theory} \cite{Mi} and Cheeger and Ebin's \textit{Comparison Theorems in Riemannian Geometry} \cite{ChEb}.  

\subsection{Basic definitions and notation}
Let $M$ be a closed manifold with a Riemannian metric $\langle\cdot,\cdot\rangle$ and induced norm $\Vert\cdot\Vert$. The length and energy of a piecewise smooth curve $\gamma:[0,T]\to M$ are given by 
\begin{equation} \begin{array}{l} 
\displaystyle L_M(\gamma)=\int\Vert\dot\gamma(t)\Vert dt \\
\displaystyle E_M(\gamma)=\int\Vert\dot\gamma(t)\Vert^2dt.
\end{array}
\end{equation}
The Cauchy-Schwartz inequality implies that $L(\gamma)^2\leq T\cdot E(\gamma)$ with equality if and only if $\gamma$ has constant speed $\Vert\dot\gamma\Vert$. A curve with constant speed $1$ is said to be parametrized by arc-length. The distance $d_M(x,z)$ between two points in $M$ is the infimum of the lengths of curves joining them and the diameter $\diam(M)$ is the maximal distance between points in $M$. A parametrized curve $\gamma:(0,T)\to M$ is a geodesic if it is locally distance minimizing. Equivalently, $\gamma$ fulfills the geodesic differential equation; hence, geodesics are smooth. We will often say that the image of a geodesic is a geodesic as well. Geodesics will usually be denoted by Greek letters $\gamma,\eta, \tau \dots$. A variation of geodesics is a smooth map $(s,t)\to\gamma_s(t)$ where $\gamma_s$ is a geodesic for all $s$. The vectorfield $\frac\D{\D s}\gamma_s(t)$ along the curve $\gamma_0$ is said to be a Jacobi field. A vector field along a geodesic is a Jacobi field if and only if it satisfies the so called Jacobi equation, a second order ordinary differential equation. In particular, the space of Jacobi fields along a geodesic is a finite dimensional vector space and every Jacobi field $J$ is determined by its initial value and derivative. Two points $x$ and $y$ in $M$ are conjugate along a geodesic arc $\gamma$ joining them if there is a nonzero Jacobi field along $\gamma$ vanishing at $x$ and $y$. 

By the Hopf-Rinow theorem, any two points in $M$ are joined by a geodesic segment whose length realizes the distance between them. Moreover, for every point $p\in M$ and for every direction $v\in T_pM$ there is a geodesic $t\mapsto\exp_p(tv)$ starting at $p$ with direction $v$. Thus we obtain the so called exponential map
$$\exp_p:T_pM\to M$$
The exponential map is a local diffeomorphism in some small neighborhood of $0\in T_pM$. The injectivity radius $\inj_p(M)$ is the maximum of those $r>0$ such that exponential map is injective on the ball $B(0,r)=\{v\in T_pM\vert\ \Vert v\Vert<r\}$. The map $p\mapsto\inj_p(M)$ is continuous and hence attains a minimum, the injectivity radius $\inj(M)$ of the manifold.
\vspace{0.2cm}

For the sake of concreteness we will always assume that the manifolds in question have injectivity radius $\inj(M)\ge 2$ and will simply denote the length and energy functions by $L$ and $E$ instead of $L_M$ and $E_M$.

\subsection{The space of broken geodesics}
Given $k\in\BN$ let $\CL_k$ be the set of piecewise geodesic curves consisting of at most $k$ edges of at most length $1$. To be more precise, elements $\gamma\in\CL_k$ are continuous curves
$$\gamma:[0,k]\to M$$
such that for all $i=0,1,\dots,k-1$ the curve $\gamma\vert_{[i,i+1]}$ is a geodesic segment with length at most $1$. When we endow $\CL_k$ with the compact open topology, the valuation map
$$\CL_k\to M^{k+1},\ \ \gamma\mapsto(\gamma(0),\dots,\gamma(k+1))$$
is continuous. Moreover, the assumption that $\inj(M)\ge 2$ implies that this map is injective and hence a homeomorphism onto its image. The interior $\CL_k^\circ$ of $\CL_k$, as a subset of $M^{k+1}$, is the set of those elements $\gamma$ consisting of geodesic arcs of length stictly less than $1$. The tangent space $T_\gamma\CL_k^\circ$ at $\gamma\in\CL_k^\circ$ is naturally identified with the space of the continuous vectorfields $J$ along $\gamma$ such that $J\vert_{[i,i+1]}$ is Jacobi for all $i=0,1,\dots,k-1$. Observe that this identification of $T_\gamma\CL_k^\circ$ is consistent with the identification of $\CL_k^0$ with an open subset of $M^{k+1}$. In particular, the later point of view induces a Riemannian metric $\langle\langle\cdot,\cdot\rangle\rangle$ on $\CL_k^\circ$. 

Given two points $p,q\in M$ set
$$\CL_k(p,q)=\{\gamma\in\CL_k(p,q)\vert\gamma(0)=p,\gamma(k)=q\}$$
Obviously $\CL_k(p,q)$ is a closed subset of $\CL_k$ homeomorphic to a closed subset of $M^{k-1}$. Moreover, from the description above we obtain that the interior $\CL_k(p,q)^\circ$ of $\CL_k(p,q)$ as a subset of $M^{k-1}$ coincides with the intersection of $\CL_k(p,q)\cap\CL_k^\circ$. In particular, the tangent space $T_\gamma\CL_k(p,q)^\circ$ of $\CL_k(p,q)^\circ$ at some curve $\gamma$ is given by the space of continuous vectorfields $J$ along $\gamma$ which vanish at $0$ and $k$ and such that $J\vert_{[i,i+1]}$ is Jacobi for all $i=0,1,\dots,k-1$.  The energy function $E(\cdot)$ is smooth in $\CL_k^\circ$ and the first variation formula asserts that the derivative of $E\vert_{\CL_k(p,q)^\circ}$ at some point $\gamma$ is given by:
\begin{equation}\label{first-variation}
d(E\vert_{\CL_k(p,q)^\circ})_\gamma(\cdot)=2 \sum_{i=1}^{k-1}\langle -\Delta\gamma(i),\cdot\rangle
\end{equation}
where $\Delta\gamma(t)=\D^+\gamma(t)-\D^-\gamma(t)$ and $\D^+\gamma(t)$ and $\D^-\gamma(t)$ are the right and left derivatives at $t$. Let $\CX$ be the negative gradient of $E\vert_{\CL_k(p,q)^\circ}$, i.e.
$$d(E\vert_{\CL_k(\gamma(0),\gamma(k))^\circ})_\gamma(\cdot)=-\langle\langle\CX_\gamma,\cdot\rangle\rangle$$
and let $\phi$ be the associated negative gradient flow
\begin{equation}\label{flow}
\phi'(t)=\CX_{\phi(t)},\ \ \phi(0)=\gamma
\end{equation} 
Observe that since the vector-field $\Delta\gamma(t)$ is smooth not only on $\CL_k(p,q)^\circ$ but on the the whole space $\CL_k^\circ$, the vector field $\CX$ and the flow $(\phi_t)$ are also smooth when considered on the whole of $\CL_k^\circ$. 

In general, gradient lines aren't defined for all $t\in\BR$ but just for some open sub-interval. However we claim that the flow $\phi$ is defined for all non-negative $t$. In fact, consider the function 
$$\lambda:\CL_k(p,q)\to[0,1],\ \ \lambda(\gamma)=\hbox{length of the longest segment in}\ \gamma$$
It is easy to check that
$$\lim_{t\to 0,\ t>0}\frac{\lambda(\phi_{\gamma}(t))-\lambda(\gamma)}t\le 0.$$
This implies that $\lambda$ is non-increasing and hence that flow lines never come close to the boundary in positive times since $\CL_k(p,q)^\circ=\{\lambda<1\}$. Thus, we have:

\begin{lem}[$\CL_k^\circ$ is a cage]\label{lem:Morse}
There is a semi-flow 
$$\phi:\CL_k^\circ\times[0,\infty)\to\CL_k^\circ,$$ $$(\gamma,t)\mapsto \phi_{\gamma}(t)$$
such that $\phi_\gamma(0)=\gamma$ and $\frac d{dt}\phi_\gamma(t)=\CX_{\phi_\gamma(t)}$ for all $\gamma$ and $t$. Moreover, the semi-flow preserves $\CL_k(p,q)^\circ$ for all $p,q\in M$.\qed
\end{lem}

We now consider the restriction of the energy function $E$ to $\CL_k(p,q)^\circ$ for some pair of points $p,q\in M$. In order to relax notation we write $E$ instead of $E\vert_{\CL_k(p,q)^\circ}$.  It follows directly from the first variation formula \eqref{first-variation} that the critical points of $E$ are precisely the geodesics of length less than $k$ joining $p$ and $q$.  

\begin{lem}[Third geodesic]\label{third-geodesic}
Assume that $\gamma_0,\gamma_1 \in \CL_k(p,q)$ are minimizing geodesics joined by a continuous curve $\gamma:[0,1] \to \CL_k(p,q)$, $s \mapsto \gamma_s$. Then there is a third geodesic $\alpha\in\CL_k(p,q)$ joining $p$ to $q$ with $E(\alpha) \leq\max_{s\in[0,1]}E(\gamma_s)$.
\end{lem}

\begin{proof}
Let $c=E(\gamma_0)=E(\gamma_1)$, $C=\max_{s\in[0,1]}E(\gamma_s)$, and assume that $\gamma_0,\gamma_1\in\CL_k(p,q)$ are the only geodesic segments joining $p$ and $q$ with energy not more than $C$.  Then for each $s\in (0,1),$ $c<E(\gamma_s)\leq C$ and by Lemma \ref{lem:Morse} and the paragraph following that lemma, $\phi_{\gamma_s}(t)$ converges to either $\gamma_0$ or $\gamma_1$ as $t \rightarrow \infty$.  Choose $t_0>0$ so that for all $s\in [0,1]$, $d(\phi_{\gamma_s}(t_0),\{\gamma_0,\gamma_1\})<d(\gamma_0,\gamma_1)/3:=d_0$.  The assumption that both $\gamma_0$ and $\gamma_1$ are minimizing implies that for $s>0$ sufficiently close to zero (resp. close to 1), $d(\phi_{\gamma_s}(t_0),\gamma_0)<d_0$ (resp. $d(\phi_{\gamma_s}(t_0),\gamma_1)<d_0$.)  Finally, for $i=0,1$, define $S_i \subset [0,1]$ by $S_i=\{s\in (0,1)\,\vert\, d(\phi_{\gamma_s}(t_0),\gamma_i)<d_0$ \}.  Then $(0,1)$ is the disjoint union of the two nonempty open sets $S_0$ and $S_1$, a contradiction.  
\end{proof}

\subsection{Totally Convex Subsets}

\begin{defi*}
A set $C$ in a complete Riemannian manifold $M$ is called totally convex if whenever $p,q \in C$ and $\eta$ is a geodesic segment from $p$ to $q$, then $\eta \subset C$.
\end{defi*}

A closed totally convex set $C\subset M$ has the structure of an embedded topological submanifold with smooth interior and possibly nonempty and nonsmooth boundary (see e.g. \cite[Chapter 8]{ChEb}).  The next result is Theorem 8.14 in \cite{ChEb}.

\begin{sat}\label{totconvex}
Let $C$ be a compact boundaryless totally convex set $C$ in $M$.  Then the inclusion $C \subset M$ is a homotopy equivalence.
\end{sat}

The idea behind the proof is to apply the negative gradient flow of the energy functional on the space $\mathcal{L}_C$ consisting of curves in $M$ with endpoints in $C$. As $C$ is totally convex, the only critical points are the constant curves into $C$.  It follows that $C\subset \mathcal{L}_C$ is a deformation retract, proving that the relative homotopy groups $\pi_i(M,C)$ vanish.  The next corollary is an easy consequence of Theorem \ref{totconvex}.

\begin{kor}\label{chord}
Assume that $\gamma \subset M$ is a closed geodesic in a closed Riemannian manifold $M$ of dimension at least two.  Then there exists a geodesic segment $\eta : [0,1] \rightarrow M$ with endpoints in $\gamma$ but not completely contained in $\gamma$.
\end{kor}

\begin{proof}
If not, then $\gamma \subset M$ is a totally convex subset and hence by Theorem \ref{totconvex}, $M$ is homotopy equivalent to $\gamma$.  This is a contradiction since the fundamental class $[M] \in H_{\dim(M)}(M,\mathbb{Z}/2\mathbb{Z})$ is a nonzero element. 
\end{proof}

\section{Main Theorems}\label{sec:light}
In this section we prove Theorems \ref{thm:sphere} and \ref{thm:surfaces}. The bulk of the work lies in proving the following technical result.

\begin{prop}\label{blocking-blaschke}
Suppose that $M$ is a closed Riemannian manifold with cross blocking. If $M$ is not a Blaschke manifold, then there is simple closed geodesic $\gamma \subset M$ of length $2\inj(M)$. 
\end{prop}
\begin{proof}
We assume that $M$ has been scaled so that $\inj(M)=2$. Choose $p \in M$ with $\inj_p(M)=\inj(M)$ and let $\cut(p)\subset T_pM$ be its cut-locus. Choose $\theta \in \cut(p)$ with $||\theta||=2$ realizing the injectivity radius. For $r>0$ and $v\in T_p M$ denote by $B(v,r) \subset T_p(M)$, the open ball with radius $r$ and center $v$. We first argue that there is an open neighborhood $U\subset T_p M$ of $\theta$ for which the restriction of $\exp_p:T_p M\to M$ to  $U \cap \cut(p)$ is one-to-one.

Indeed, if this were not the case, then the restriction of $\exp_p$ to $B(\theta,r) \cap \cut(p)$ is not one-to-one for each $r>0$.  Fix a positive $\epsilon'$ smaller than $\frac{1}{2}$.  By continuity of the exponential map and the distance function in $M$, there is a sufficiently small $r_0>0$ so that for all $\theta_0,\theta_1 \in B(\theta,r_0)$ we have that 
$$d_M(\exp_p(\frac{\theta_0}{2}), \exp_p(\frac{\theta_1}{2}))<\frac{\epsilon'}{2}$$  
Let $\epsilon<\min\{\epsilon',r_0, \diam(M)-2\}$ and choose $\theta_0,\theta_1 \in B(\theta,\epsilon)\cap \cut(p)$ with $\exp_p(\theta_0)=\exp_p(\theta_1):=q$.  Define $\gamma_i:[0,4] \to M$ by $\gamma_i(t):=\exp_p(t\frac{\theta_i}{4})$ for $i=0,1$.  Note that both $\gamma_0$ and $\gamma_1$ are minimizing geodesics between $p$ and $q$ with $L(\gamma_i)\leq 2+\epsilon$ for $i=0,1$. We consider the curve 
$$\sigma_p:[0,1] \to T_pM,\ \ \sigma_p(s)=(1-s)\frac{\theta_0}{2}+s\frac{\theta_1}{2}$$
in the tangent space to $M$ at $p$ and its image under the exponential map
$$\sigma:[0,1] \to M,\ \ \sigma(s)=\exp_p(\sigma_p(s)).$$
For each $s\in [0,1]$, we have that 
$$d_M(q,\sigma(s))\leq d_M(q,\sigma(0))+d_M(\sigma(0),\sigma(s)) \leq \frac{2+\epsilon}{2}+\frac{\epsilon'}{2} < 1+\epsilon'<2.$$
Therefore, there is a unique curve $\sigma_q:[0,1]\to B(0,2)\subset T_qM$ with $\exp_q(\sigma_q(s))=\sigma(s)$. For $s\in[0,1],$ define the one paramater family of curves $s \mapsto \gamma_s$ by 
$$\gamma_{s}(t) = \left\{ \begin{array}{rl}
         \exp_p(t\frac{\sigma_p(s)}{2}), & \mbox{ for } t \in [0,2]\\
       \exp_q((4-t)\frac{\sigma_q(s)}{2}), & \mbox{ for } t \in[2,4].\end{array} \right . $$ 

\begin{figure}[htbp]
\begin{center}
\includegraphics[scale=0.5,angle=270]{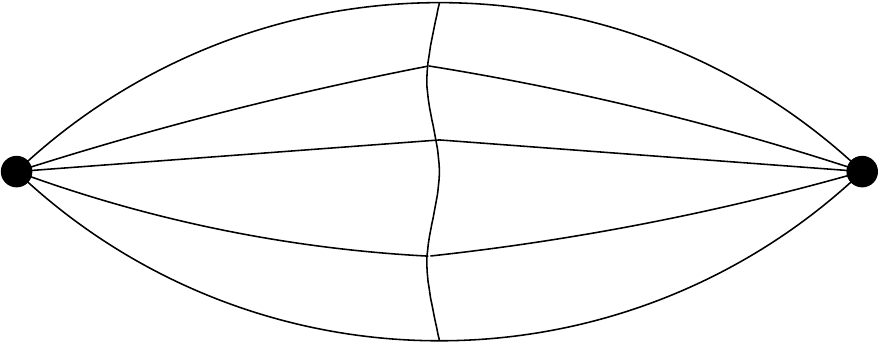}
\end{center}
\caption{The variation $\gamma_s$ interpolating by not much longer curves between the geodesics $\gamma_0$ and $\gamma_1$.}
\end{figure}

It is easy to check that $L(\gamma_s\vert_{[i,i+1]}) <1$ for all $s\in[0,1]$ and $i=0,\dots,3$ so that this family defines a continuous curve $\gamma:[0,1] \to \CL_4(p,q)^\circ$, $s \mapsto \gamma_s$, connecting $\gamma_0$ and $\gamma_1$.  One also checks easily that for each $s\in[0,1]$, the curve $\gamma_s$ has at most energy $(1+ \epsilon')^2$ so that by Lemma \ref{third-geodesic}, there is a third geodesic $\alpha\in \CL_4(p,q)^\circ$ joining $p$ to $q$ with $E(\alpha) \leq (1+ \epsilon')^2$.  It follows that $L(\alpha) \leq 2+ 2\epsilon' < 3$.  Note that since each of $\alpha, \gamma_0$, and $\gamma_1$ have length strictly less than 4, no two can intersect in their interiors without contradicting $\inj(M)=2$.  Hence, $b_M(p,q) \geq 3$, a contradiction to cross blocking since $d_M(p,q)\leq 2+\epsilon < \diam(M)$.

We have proved that there is some open neighborhood $U\subset T_p M$ of $\theta$ such that the restriction of $\exp_p$ to $U \cap \cut(p)$ is one-to-one. From now on, let $U$ be such a neighborhood. 

We argue next that there are at least two distinct unit speed minimizing geodesics 
$\gamma_0 ,\gamma_1:[0,2] \rightarrow M$ joining $p$ and $q:=\exp_p(\theta)$ (and hence exactly two by the cross blocking condition). Define 
\begin{align*}
&r_p:\partial \overline{B}(0,1) \to (0,\diam(M)]\\
&r_p(v)=\sup\{t\in(0,\diam(M)]\vert d_M(p,\exp_p(tv))=t\}
\end{align*}
It is well-known that the function $r_p$ is continuous. Hence, the function
$$i_p:\overline{B}(0,1)\setminus\{0\}\to T_p M,\ \ \ i_p(x)=r_p(\frac{x}{||x||})x$$ 
is continuous as well. Therefore, $i_p^{-1}(U)$ is an open subest of $\frac{\theta}{2}$ in $\overline{B}(0,1)$.  Choose $\delta>0$ sufficiently small so that the set $V_{\delta}:=\overline{B}(\frac{\theta}{2},\delta) \cap \overline{B}(0,1)$ is contained in $i_p^{-1}(U_{\theta})$.  Note that $V_{\delta}$ is homeomorphic to a basic closed set of $0$ in the upperhalf space $\mathbb{R}^n_{x_n \geq 0}$ and that the map $\exp_p\circ i_p$ is continuous and one-to-one on $V_{\delta}$.  Hence, $\exp_p(i_p(V_\delta))$ does not cover an entire neighborhood of $q$ so that we find a sequence of points $q_i \in M-\exp_p(i_p(V_{\delta}))$ converging to $q$.  For each, $i$, let 
$$\eta_i:[0,2] \to M$$ 
be a minimizing geodesic joining $p$ to $q_i$ and define $\gamma_0:[0,2] \to M$ by $\gamma_0(t)=\exp_p(t \frac{\theta}{2})$.  Up to passing to a subsequence the minimizing geodesics $\eta_i$ converge to a second unit speed geodesic $\gamma_1:[0,2] \to M$ joining $p$ to $q$.

Next we argue that $\gamma_0$ and $\gamma_1$ together form a closed geodesic. If not, then either $\dot{\gamma_0}(0) \neq  -\dot{\gamma_1}(0)$ or $\dot{\gamma_0}(2) \neq  -\dot{\gamma_1}(2)$.  We assume the latter, the former case being handled symmetrically.  Fix a positive $\epsilon<1$ and choose $v \in T^1_q M$ making obtuse angle with both $\dot{\gamma_0}(2)$ and $\dot{\gamma_1}(2)$. Note that for all sufficiently small $s$, the distance between the points $\gamma_i(2-\epsilon)$ and $sigma(s)=\exp_q(sv)$ is less than one and in particular they are connected by a unique minimizing geodesic segment $\sigma^i_s:[0,1] \to M$. By the first variation formula the energy $E(\sigma^i_s)$ is strictly decreasing for sufficiently small $s$. Fix $s_0<\epsilon$ positive and small enough such that $E(\sigma^i_{s_0})<E(\sigma^i_0)=\epsilon^2$.

For $i=0,1$ define broken geodesics $\alpha_i:[0,3] \to M$ by         
$$ \alpha_i(t) = \left\{ \begin{array}{rl}
         \gamma_i(\frac{2-\epsilon}{2}t), & \mbox{ for } t \in [0,2]\\
       \sigma^i_{s_0}(t-2) & \mbox{ for } t \in[2,3]\end{array} \right. $$ 
The curves $\alpha_0,\alpha_1$ belong to $\CL_3(p,\exp_q(s_{0}v))^\circ$ and have at most energy 
$$E(\alpha_i)\le \frac{(2-\epsilon)^2 +2\epsilon^2}{2}<2.$$
Since $d_M(p,\exp_q(s_{0}v))<2$ is less than the injectivity radius, the points $p$ and $\exp_p(s_{0}v)$ are connected by a unique geodesic segment $\alpha$ shorter than $2$. The uniqueness of $\alpha$ implies that the flow lines $\tau\mapsto\phi_{\alpha_i}(\tau)$ of the flow provided by Lemma \ref{lem:Morse} and starting in $\alpha_0$ and $\alpha_1$ respectively converge to $\alpha$ with $\tau\to\infty$.  We conclude that $\alpha_0$ and $\alpha_1$ are homotopic through piecewise geodesics with three segments having energy not more than $\frac{(2-\epsilon)^2 +2\epsilon^2}{2}.$ See figure 4.

\begin{figure}[htbp]
\begin{center}
\includegraphics[scale=0.5,angle=270]{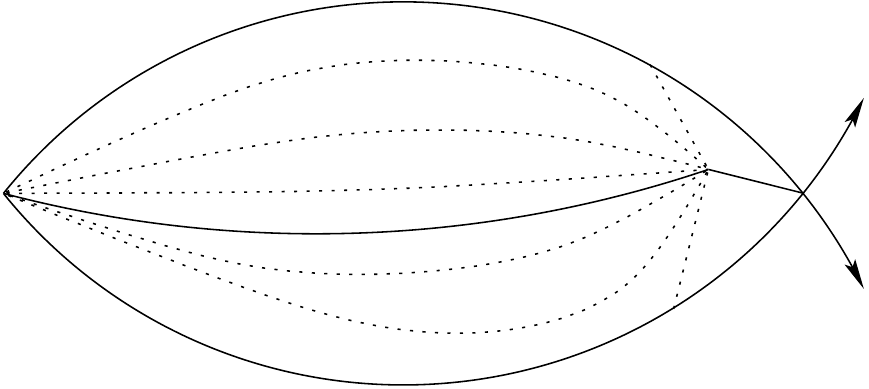}
\end{center}
\caption{Flowing $\alpha_0$ and $\alpha_1$ to the geodesic $\alpha$.}
\end{figure}

Similarly, the once broken geodesics joining $\gamma_i(2-\epsilon)$ to $q$ defined by concatenating $\sigma^{i}_{s_0}$ with $\sigma\vert_{[0,s_0]}$ traversed in the opposite direction are homotopic to $\gamma_i\vert_{[2-\epsilon,2]}$ through once broken geodesics of total energy not more than $2\epsilon^2$.  Combining these homotopies with those between $\alpha_0$ and $\alpha_1$ yields a continuous curve $\gamma:[0,1] \to \CL_4(p,q)^\circ$, $s \mapsto \gamma_s$, joining $\gamma_0$ and $\gamma_1$ with $\max_{s\in[0,1]} E(\gamma_s) \leq \frac{(2-\epsilon)^2+4\epsilon^2}{2}.$  By Lemma \ref{third-geodesic}, there is a third geodesic $\beta:[0,4]\to M$ joining $p$ to $q$ with $E(\beta)<\frac{(2-\epsilon)^2+4\epsilon^2}{2}$.  One easily checks that $L(\beta) <4$. Therefore, $\beta$ cannot intersect $\gamma_0$ or $\gamma_1$ in their interiors without contradicting $\inj(M)=2$.  Hence $b_M(p,q) \geq 3$, contradicting cross blocking since $d(p,q)=2<\diam(M)$.  

We obtain that $\dot{\gamma_0}(0)= -\dot{\gamma_1}(0)$ and $\dot{\gamma_0}(2)= -\dot{\gamma_1}(2)$, completing the proof of Proposition \ref{blocking-blaschke}.
\end{proof}

Next, we prove Theorems \ref{thm:sphere} and \ref{thm:surfaces}.
\begin{named}{Theorem \ref{thm:sphere}}
A closed Riemannian manifold $M$ has cross and sphere blocking if and only if $M$ is isometric to a round sphere.
\end{named}

\begin{proof}
We first scale the metric on $M$ so that $\inj(M)=2$.  To begin with we claim that $M$ is a Blaschke manifold. Otherwise there is simple closed geodesic $\gamma \subset M$ with 
\begin{equation}\label{eq:inj-diam}
L(\gamma)=2\inj(M)<2\diam(M)
\end{equation}
by Proposition \ref{blocking-blaschke}. By Corollary \ref{chord}, there is a geodesic segment $\eta:[0,1] \to M$ with end-points in $\gamma$ not entirely contained in $\gamma$. Up to replacing $\eta$ by a subsegment whose end-points are again in $\gamma$ we can assume that the interior of $\eta$ is disjoint from $\gamma$. Let $x$ and $y$ be the end-points of $\eta$. If $x=y$ then $\gamma$ and $\eta$ are two light rays from $x$ to itself with disjoint interior. Hence one needs at least two points to block $x$ from itself contradicting the assumption that $M$ has sphere blocking. Assume now that $x\neq y$. Then $\eta$ and the two subsegments of $\gamma$ connecting $x$ and $y$ are three light rays with disjoint interior. This implies that $x$ and $y$ have at least blocking number $b_M(x,y)\ge 3$. Since $M$ is assumed to have cross blocking we obtain that $x$ and $y$ are at distance $\diam(M)$ and hence $\gamma$ has at least length $2\diam(M)$ contradicting \eqref{eq:inj-diam}. 

We have proved that $M$ is Blaschke. As mentioned in the introduction, Theorem \ref{thm:sphere} follows now from \cite[Corollary 3.7]{LaSc} where it was shown that Blaschke manifolds with sphere blocking are isometric to round spheres. 
\end{proof}

\begin{named}{Theorem \ref{thm:surfaces}}
A closed Riemannian surface $M$ has cross blocking if and only if $M$ is isometric to a constant curvature sphere or projective plane. 
\end{named}

\begin{proof}
Assume that $M$ is a closed Riemannian surface with cross blocking that does not have constant positive curvature.  By \cite{Be}, $M$ is not Blaschke.  By Proposition \ref{blocking-blaschke}, there is a simple closed geodesic $\gamma \subset M$ of length $2\inj(M)<2\diam(M).$  

We first claim that $\gamma$ must generate $\pi_1(M)$.  To see this, fix $p \in  \gamma$ and an essential map $f:([0,1],\{0,1\})\rightarrow (M,p)$ representing an element in $\pi_1(M,p)$ not in the subgroup generated by $\gamma$.  Let $\epsilon>0$ be small, choose a point $q\in \gamma\cap B(p,\epsilon)$ different from $p$, and let $\gamma'$ denote the subsegment of $\gamma$ joining $p$ to $q$ of length less than $\epsilon$.  Concatenating $f$ with $\gamma'$ yields a map $f':[0,1]\rightarrow M$ with $f(0)=p$, $f(1)=q$, and with the property that any other curve joining $p$ to $q$ homotopic to $f'$ relative to the endpoints must have energy strictly greater than $\epsilon$.  A curve $\tau$ minimizing energy in this homotopy class is a geodesic segment joining $p$ to $q$ with image not entirely contained in $\gamma$.  Up to passing to a subsegment of $\tau$ with distinct endpoints $p'$ and $q'$, we may assume that the interior of $\tau$ never intersects $\gamma$ .  But then $\tau$, and the two subsegments of $\gamma$ joining $p'$ to $q'$ are three light rays with distinct interiors.  As $M$ is assumed to be cross blocked, $d(p',q')=\diam(M)$, a contadiction, completing the proof that $\gamma$ generates $\pi_1M$.  In particular $M$ is diffeomorphic to either $S^2$ or $\mathbb{R}\mathbb{P}^2$.

Next, assume that $M$ is diffeomorphic to $S^2$.  Then $\gamma$ bounds a Riemannian 2-disc $\mathbb{D}\subset S^2$.  By \cite{HaSc}, there is a geodesic segment $\tau:[0,1] \rightarrow \mathbb{D}$ making right angles at both ends with $\gamma$.  In particular, the endpoints of $\tau$ are distinct.  Hence, $\tau$ and the two subsegments of $\gamma$ joining $\tau(0)$ to $\tau(1)$ are three light rays between these points with distinct interiors.  Again, as $M$ is cross blocked, the distance between these endpoints is $\diam(M)$, a contradiction.

Thus, $M$ is diffeomorphic to $\mathbb{R}\mathbb{P}^2$.  Lift $\gamma$ to a closed geodesic $\tilde{\gamma} \subset S^2$.  Let $A:S^2 \rightarrow S^2$ be the order two covering transformation corresponding to $\gamma$ and for $p \in \tilde{\gamma}$, let $p'=A(p)$.  We shall say that such a pair of points $p,p' \in \gamma$ are an antipodal pair.  Note that the same reasoning as in the above two paragraphs shows that any geodesic segment $\tau:[0,1]\rightarrow S^2$ with endpoints in $\tilde{\gamma}$ and interior disjoint from $\tilde{\gamma}$ satisfies $\tau(0)=\tau(1)$ or $\tau(1)=\tau(0)'$.  It follows easily that $\tilde{\gamma}$ has no transversal self-intersections and that each pair of subsegments of $\tilde{\gamma}$ joining antipodal pairs are minimizing.  Fix a hemisphere $\Sigma$ bounded by $\tilde{\gamma}$.  We will next prove that $\Sigma$ is isometric to a constant curvature hemisphere, contradicting the assumption that $M$ does not have constant curvature, and completing the proof of the theorem.  

For $p \in \tilde{\gamma}$, let $T_p^+(\tilde{\gamma})\subset T_p S^2$ denote the set of unit tangent vectors based at $p$ either tangent to $\tilde{\gamma}$ or pointing into the hemisphere $\Sigma$.  

By \cite{HaSc}, there is a constant speed paramaterized geodesic segment $\tau:[0,1] \rightarrow \Sigma$ making right angles with $\tilde{\gamma}$ at both endpoints.  Let $p=\tau(0)$ and $L=\length{\tau}$.  Then by the above remarks, $\tau(1)=p'$.   Note that $p'$ is conjugate to $p$ along $\tau$ for otherwise there are geodesic segments arbitrarily close to $\tau$ joining $p$ to a point $p'' \in  \tilde{\gamma}$ distinct from but arbitrarily close to $p'$.  Let $\Conj(p)\subset T_p S^2$ denote the tangential conjugate locus to $p$ and $C$ the component containing $\dot{\tau}(0)$.  By work of Warner in \cite{Wa}, $C$ is a smooth $1$-submanifold of $T_pS^2$ transverse to the radial directions.

Let $v:=\frac{\dot{\tau}(0)}{|| \dot{\tau}(0) ||} \in  T_p^+(\tilde{\gamma})$ and note that for all vectors $v'$ sufficiently close to $v$ in $T_p^+(\tilde{\gamma})$, the geodesic ray $\tau_{v'}(t)=\exp_p(tv')$ crosses $\tilde{\gamma}$ in a small neighborhood of $p'$ at time close to $L$.  By the above remarks, the point of intersection of each such ray must be $p'$ and the antipodal pair $p$ and $p'$ are conjugate along each such ray.  Let $U \subset T_p^+(\tilde{\gamma})$ be the largest open interval around $v$ with the property that each ray in a direction through $U$ first leaves $\Sigma$ through the point $p'$.  Note that the times the rays in directions from $U$ leave $\Sigma$ through $p'$ vary smoothly with $v' \in U$.  This follows since $p$ and $p'$ are conjugate along each such ray and since $C$ is a smooth curve.  By the first variation formula, they actually all leave at time exactly $L$.  It follows that $U$ is closed and hence that $U= T_p^+(\tilde{\gamma})$ and $L=2\inj(M)=\length{\tilde{\gamma}}/2$.  

It now follows that for each $q \in \tilde{\gamma}$ sufficiently close to $p$, there is a geodesic ray entering $\Sigma$ from $q$ and leaving $\Sigma$ at a point in $\tilde{\gamma}$ close to $p'$.  By repeating the argument in the last paragraph, it follows that for $q$ sufficiently close to $p$, every geodesic entering $\Sigma$ at $q$ first exits $\Sigma$ at its antipodal point $q'$ at time exactly $L$.  Let $U'$ denote the largest open interval around $p$ in $\tilde{\gamma}$ with the property that every ray entering $\Sigma$ from a point in $U$ exists $\Sigma$ at its antipode at time $L$.  Then $U'$ is clearly closed, whence $U'=\tilde{\gamma}.$  By \cite{Ba}, $\Sigma$ is a round hemisphere, completing the proof.
\end{proof}

{\tiny \sc Benjamin Schmidt, University of Chicago, schmidt@math.uchicago.edu}

{\tiny \sc Juan Souto, University of Michigan, jsouto@umich.edu}

\end{document}